\documentclass[10pt,a4paper]{article}
\usepackage[final]{optional}
\usepackage[british,american]{babel}

\usepackage{latexsym}
\usepackage{amsxtra}
\usepackage{mathrsfs}

\usepackage{amsfonts, amsmath}
\usepackage{amssymb,amsthm}

\usepackage[usenames]{color}

\usepackage{setspace}

\theoremstyle{plain}
\newtheorem{lemma}{Lemma}[section]
\newtheorem{thm}[lemma]{Theorem}

\newtheorem{cor}[lemma]{Corollary}

\theoremstyle{definition}
\newtheorem{defn}[lemma]{Definition}

\newtheorem{ex}[lemma]{Example}
\newtheorem{rmk}[lemma]{Remark}

\numberwithin{equation}{section}

\newcommand{\be}{\begin{equation}}
\newcommand{\ee}{\end{equation}}

\newcommand{\Ric}{\operatorname{Ric}}
\newcommand{\Rm}{\operatorname{Rm}}

\newcommand{\torus}{\mathbb{T}^2}
\newcommand{\tee}{\texttt{t}}
\newcommand{\ubar}{\overline{u}}
\newcommand{\MM}{M_1 \sqcup M_2}

\newcommand{\veca}{\textbf{a}}
\newcommand{\vecb}{\textbf{b}}

\newcommand{\vecc}{\textbf{c}}
\newcommand{\x}{\textbf{x}}

\newcommand{\sphere}{\mathbb{S}}
\newcommand{\R}{\ensuremath{{\mathbb{R}}}}

\newcommand{\vol}{\text{Vol}}
\newcommand{\del}{\partial}
\newcommand{\Lip}{\operatorname{Lip}}

\DeclareMathOperator{\tr}{tr}

\title{ Super Ricci flow for disjoint unions 
\\
}
\author{Sajjad Lakzian, Michael Munn}

\date{\today}
\begin{document}
\maketitle

\begin{abstract}
In this paper we consider compact, Riemannian manifolds $M_1, M_2$ each equipped with a one-parameter family of metrics $g_1(t), g_2(t)$ satisfying the Ricci flow equation. Motivated by a characterization of the super Ricci flow developed by McCann-Topping in \cite{McCT2010}, we introduce the notion of a super Ricci flow for a family of distance metrics defined on the disjoint union $\MM$. In particular, we show such a super Ricci flow property holds provided the distance function between points in $M_1$ and $M_2$ evolves by the heat equation. We also discuss possible applications and examples. 
\end{abstract}

\section{Introduction}\label{Section-Results}
For $i = 1,2$, let $M_i$ be a compact $n$-dimensional Riemannian manifold equipped with a smooth family of metrics $g_i(t)$ satisfying the Ricci flow equation introduced by Hamilton \cite{H1983} 
\be\label{RF}\frac{\partial g_i(t)}{\partial t} = -2 \Ric({g_i(t)}),\ee
for  $t \in [0, T_i)$. The short-time existence and uniqueness of  solutions was demonstrated in \cite{H1983} and we denote $T = \min(T_1, T_2)$. In this note, we consider the disjoint union $M_1 \sqcup M_2$ equipped with a one-parameter family of metrics $D^t$, for $t \in [0, T)$, so that  $\left(M_1 \sqcup M_2, D^t\right)$ is a complete, compact metric space whose metric is compatible with the evolving metrics $g_i(t)$; i.e. for $i=1,2$
\be \label{restricted d} \left. D^t \right|_{M_i} = d_{g_i(t)}, \ee
where $d_g$ denotes the distance metric induced by the Riemannian metric $g$. Following \cite{McCT2010}, we generalize the characterization of super Ricci flow solutions for an individual family of smooth metrics, say for $(M_1, g_1(t))$ or $(M_2, g_2(t))$, to the family of metric spaces $(M_1 \sqcup M_2, D^t)$ as follows

\begin{defn}\label{LMsuperRF}
With $M_1$ and $M_2$ as above, a family of metrics $D^t$ on $M_1 \sqcup M_2$, for $t \in [0,T)$, is called a {\em super Ricci flow of the disjoint union $M_1 \sqcup M_2$} provided whenever $0 <a<b<T$ and $u(x,t): M_1 \sqcup M_2 \times (a,b)\to \mathbb{R}$ is a solution to the heat equation on $M_1 \sqcup M_2$, then
\be \Lip(u,t):= \sup_{\substack{x \neq  y \\x,y \in \MM}}  \frac{|u(x,t) - u(y,t)|}{D^t(x,y)} ~\text{ is non-increasing in }~ t.\ee
\end{defn}
In Section \ref{Section-Background}, we recall work of von Renesse-Sturm \cite{vRS2005} to clarify precisely the Laplacian we are using on $\MM$ and exactly what it means for $u(x,t)$ to satisfy the heat equation for such a disconnected space (see Definition \ref{disjoint heat soln} and the discussion therein). 

Furthermore, we show that,

\begin{thm}\label{MainThm}
For $i=1,2$, let $M_i$ be a compact, oriented $n$-dimensional manifold equipped with a smooth family of metrics $g_i(t)$ satisfying the Ricci flow equation (\ref{RF}) for $t \in [0, T_i)$ and let $T = \min(T_1, T_2)$. Consider the family of metric spaces $(M_1 \sqcup M_2, D^t)$ and  suppose that for $t \in (0,T)$, 
\be  \label{d evolution} \frac{\partial}{\partial t}D^t(x,y) \geq \Delta_{M_1^t \times M_2^t}D^t(x,y), \quad  \text{for }~x \in M_1, y \in M_2, \ee
where $\Delta_{M_1^t \times M_2^t}$ denotes the Laplacian on $(M_1,g_1(t)) \times (M_2,g_2(t))$. Then the family of metrics $D^t$ is a super Ricci flow of $M_1 \sqcup M_2$.
\end{thm}

\begin{rmk} The statement of Theorem \ref{MainThm} can be phrased slightly more generally in that $(M_1, g_1(t))$ and $(M_2, g_2(t))$ need only be {\it supersolutions} to the Ricci flow equation; i.e. $g_i(t)$ are super Ricci flows (see Definition \ref{Def-superRF}) on $M_i$, $i=1,2$. Indeed, the proof requires only this slightly weaker assumption. 
\end{rmk}

\begin{rmk} Note that condition (\ref{d evolution}) alone isn't enough to guarantee that the family of distance functions $D^t$ between $M_1$ and $M_2$ remain distance functions for all $t$. This is because it is possible that the triangle inequality may fail at certain times $t>0$, particularly if either $M_1$ or $M_2$ has highly negative sectional curvature. However, in the statement of Theorem \ref{MainThm} we implicitly restrict our attention to only those families $D^t$ which in fact are distance functions. In Section \ref{Section-Examples} we give simple constructions which verify that the class of such distance functions on $\MM$ is nonempty.
\end{rmk}

\textit{Acknowledgements.} This research was sponsored in part by the National Science Foundation Grants OISE \#0754379 and DMS \#1006059. We would like to thank the Graduate Center at CUNY where part of this work was completed and Prof. Sormani for her interest and support. In addition, MM would like to thank the Mathematics Institute of the University of Warwick and Peter Topping for their hospitality as this work was completed. SL also thanks Dan Knopf for helpful discussions during his visit to U Texas at Austin.

\subsection{Motivation}\label{Section-Intro}
We now say a few words of context for Theorem \ref{MainThm} and give possible perspectives for considering such a family of metric spaces $(\MM, D^t)$.

A primary advantage of Theorem \ref{MainThm} is that the nature of condition (\ref{d evolution}) is purely metric and gives a sufficient condition for a family of distance metrics on the set $\MM$ to evolve in a way that is compatible with the smooth evolution of the Ricci flow for the Riemannian metrics on $M_1$ and $M_2$. This metric perspective allows for a more broad description of solutions to the Ricci flow (or in this case, super solutions to the Ricci flow) which can persist {\it through} the development of singularities provided one has knowledge of the metric after the singular time. 

Given $M^n$ a compact, $n$-dimensional Riemannian manifold and $g(t)$ a family of smooth metrics evolving by (\ref{RF}), we say a finite time singularity develops at time $T$ if this family cannot be extended beyond $T<\infty$. Standard long time existence theorems imply that such finite time singularities develop if and only if the Riemann curvature tensor $\Rm$ blows up as $t \nearrow T$; i.e.
\be \limsup_{t\nearrow T} ~\sup _{x\in M} \left|\Rm(x, t)\right|_{g(t)} = \infty.\ee
Sesum \cite{S2005} improved this by showing that a finite-time singularity occurs if and only if 
\be \limsup_{t\nearrow T} ~\sup _{x\in M} \left|\Ric(x, t)\right|_{g(t)} = \infty.\ee
In some sense, the formation of such singularities is a `typical' property of the Ricci flow. Indeed, it follows from the parabolic maximum principle that if the scalar curvature $R$ satisfies $R \geq \alpha >0$ at time $t = 0$, then a finite time singularity must develop for $T \leq \dfrac{n}{2\alpha}$. As a result, the study of the formation of singularities remains an intensely studied aspect of the Ricci flow and geometric evolution equations in general. 

%The first rigorous examples of finite time singularities were constructed by Miles Simon in \cite{S2000} These examples are \textit{local singularities} in that are defined on proper compact subsets of noncompact warped product metrics on $\R \times_f \sphere^n$. Another class of noncompact examples was constructed by Feldman-Ilmanen-Knopf in \cite{FIK2003} for a family of metrics $g(t)$ that are complete $U(n)$-invariant shrinking gradient Kahler-Ricci solitons on the holomorphic line bundle $L^{-k}$ over $\mathbb{CP}^{n-1}$ with twisting number $k \in \{1,..., n-1 \}$.

Angenenet-Knopf  were the first to give examples of finite time singularities for compact manifolds \cite{AK2004, AK2007}, although certain constructions did exist for local singularities on noncompact manifolds  \cite{S2000, FIK2003}. Specifically, Angenenet-Knopf examined the behavior of the metrics on topological spheres $\mathbb{S}^{n+1}$ evolving by the Ricci flow and showed that when the initial metric $g_0$ is sufficiently pinched, the Ricci flow will develop a neck-pinch singularity. 

A \textit{neck-pinch singularity} is a special kind of local Type I singularity and (except for the round sphere shrinking to a point) is arguably the best known and simplest example of a finite-time singularity that can develop through the Ricci flow. More precisely, a solution $(M^{n+1}, g(t))$ of the Ricci flow develops a \textit{neck pinch} at time $T < \infty$ if there exists a time-dependent family of proper open subsets $N(t) \subset M^{n+1}$ and diffeomorphisms $\phi_t: \R \times \sphere^n \to N(t)$ such that $g(t)$ remains regular on $M^{n+1} \setminus N(t)$ and the pullback $\phi_t^*\left(\left.g(t)\right|_{N(t)}\right)$ on $\R \times \sphere^n$ approaches the ``shrinking cylinder'' soliton metric
$$ds^2 + 2(n-1)(T-t) g_{can}$$
in $\mathcal{C}^{\infty}_{loc}$ as $t \nearrow T$, where $g_{can}$ denotes the round metric on the unit sphere $\sphere^n$. In \cite{AK2004}, the authors show how these neck pinch singularities arise for a class of rotationally symmetric initial metrics on $\sphere^{n+1}$. In \cite{AK2007}, they derive detailed asymptotics of the profile of the solution near the singularity as well as comparable asymptotics for fully general neck pinches whose initial metric need need not be rotationally symmetric.

Later in  \cite{ACK}, Angenent-Caputo-Knopf extended this work by constructing smooth forward evolutions of the Ricci flow starting from initial singular metrics which arise from rotationally symmetric neck pinches on $\sphere^{n+1}$ by passing to the limit of a sequence of Ricci flows with surgery. Together \cite{AK2004, AK2007, ACK}  provide a framework (albeit in the restrictive context of rotational symmetry) for developing the notion of a `canonically defined Ricci flow through singularities' as conjectured by Perelman in \cite{P2002}. Up to this point, continuing a solution of the Ricci flow past a singular time $T < \infty$ required surgery and a series of carefully made choices so that certain crucial estimates remain bounded through the flow. A complete `canonical Ricci flow through singularities' would avoid these arbitrary choices and would be broad enough to address all types of singularities that arise in the Ricci flow.

Returning now to the current paper, our motivation follows from this work of Angenent-Knopf and Angenent-Caputo-Knopf, though our result allows for application in a more general context. Since the smooth forward evolution described in \cite{ACK} performs a topological surgery on $\mathbb{S}^{n+1}$ at the singular time $T$, all future times will consist of two disjoint smooth Ricci flows on a pair of manifolds. Furthermore, although the metric $g(t)$ is no longer a smooth Riemannian metric at the singular time $t=T$, the space $\mathbb{S}^{n+1}$ does retain the structure of a metric space with distance metric denoted $d_T$ arising  from the convergence of the distance metrics $d_t$ on $(\sphere^{n+1}, g(t))$ through the evolution. As metric spaces, these spaces converge in the Gromov-Hausdorff sense as well,
\be\lim_{t \nearrow T}d_{GH}\left((\mathbb{S}^{n+1}, d_t), (\mathbb{S}^{n+1}, d_T) \right)  = 0.\ee
Our Theorem \ref{MainThm} gives a metric context in which to frame the evolution of the Ricci flow for $t > T$, {\em after} this singularity develops. 

The remainder of this paper is organized as follows. In Section \ref{Section-Examples}, we give some simple examples of metric constructions for the super Ricci flow for disjoint unions of two smooth Riemannian manifolds. In particular, we consider the situation when $M_1 \cong M_2$ and consider the case of the flat torus and the round sphere. In Section \ref{Section-Background}, we recall the characterization of the super Ricci flow given by McCann-Topping for compact Riemannian manifolds which motivates our Definition \ref{LMsuperRF}. Also, we recall a construction of von Renesse-Sturm \cite{vRS2005} and use a generalization of the Trotter-Chernov product formula for time dependent operators to describe solutions to the heat equation on the  disjoint union $\MM$. With these definitions and context in place, we then prove Theorem \ref{MainThm} in Section \ref{Section-Proof/Consequences} and give implications.

\section{Examples}\label{Section-Examples}
To better illustrate the content of Theorem \ref{MainThm} we mention a few simple examples. In general, for $(M_1, g_1(t))$ and $(M_2, g_2(t))$ as in Section \ref{Section-Results}, a family of distance metrics $D^t$ on $\MM$ for $t \in [0,T)$ is a family of non-negative functions
\be D^t: \MM \times \MM \to \mathbb{R}\ee
such that the following properties hold. For $\veca, \vecb, \vecc \in \MM$, and all $t \in [0,T)$,
\begin{itemize}
\item $D^t(\veca, \vecb) = 0$ if and only if $\veca = \vecb$
\item $D^t(\veca, \vecb) = D^t(\vecb, \veca)$
\item $D^t(\veca, \vecb) \leq D^t(\veca, \vecc) + D^t(\vecc, \vecb)$
\end{itemize}
Thus, we require these properties to hold implicitly in the statement of Theorem \ref{MainThm}. Note, however, that the metric $D^t$ is not an intrinsic distance as $\MM$ is disconnected. 

Consider the case where $(M_1, g_1(0)) \cong (M_2, g_2(0))$ and thus $g_1(t) = g_2(t)$ for all $t$ satisfying (\ref{RF}) by uniqueness. Set $D^t$ on $\MM$ to be 
\be \label{Defn-D^t}D^t(\veca, \vecb) = \begin{cases} 
d_{g_i(t)}(\veca, \vecb), \hspace{.8in} \text{ if } \veca, \vecb \in M_i\\
\\
\sqrt{L^2(t) + d_{g_i(t)}^2(\phi(\veca), \vecb)}, \text{ if } \veca \in M_1, \vecb \in M_2 ~\text{ or } ~\veca \in M_2, \vecb \in M_1, 
\end{cases}
\ee
where $\phi: M_1 \to M_2$ is the identity map and $L(t)$ depends only on $t$. Note that each of the properties for $D^t$ to be a distance function hold naturally in this construction. 

Now letting $d_t$ denote $d_{g_i(t)}$ where there is no confusion since $g_1(t) = g_2(t)$ and considering $d_t$ and $D^t$ as maps on $M_1 \times M_2$, we can relate $\Delta_{M_1^t \times M_2^t}D^t$ to $\Delta_{M_1^t \times M_2^t}d_t$. Computing in local coordinates we have
\begin{eqnarray}
\Delta (D^t)^2&=& \frac{1}{\sqrt{|g|}}\del_i \left(\sqrt{|g|} g^{ij} \del _j(D^t)^2\right)\\
&=&  \frac{1}{\sqrt{|g|}}\del_i \left(\sqrt{|g|} g^{ij} 2D^t\del _jD^t\right)\\
&=& 2 g^{ij} \del_i D^t\del_j D^t + 2D^t \Delta D^t;
\end{eqnarray}
and, directly we find $\Delta (D^t)^2 = \Delta \left( L^2(t) + d_t^2\right) = \Delta d_t^2 = 2\left | \nabla d_t\right|^2 + 2d_t \Delta d_t$, since in this simple example we assume that $L(t)$ depends {\it only} on $t$.
Thus, 
\begin{eqnarray}
d_t \Delta_{M_1^t \times M_2^t} d_t &=& \frac{1}{2}\Delta_{M_1^t \times M_2^t} (D^t)^2 - \left | \nabla^{M_1^t \times M_2^t} d_t\right|_{g(t)}^2 \\
&=& g^{ij}_t \del_i D^t \del_j D^t + D^t \Delta_{M_1^t \times M_2^t} D^t - \left| \nabla^{M_1^t \times M_2^t} d_t\right|_{g(t)}^2.
\end{eqnarray}

%&=&g^{ij}_t \del_i D^t \del_j D^t + D^t \Delta_{M_1^t \times M_2^t} D^t - g^{ij}_t \del_i d_t \del_j d_t,
%where in the last line we used the fact that, for a local basis of tangent vectors $\del_i$ on $M_1 \times M_2$,
%\begin{eqnarray}
%\left| \nabla^{M_1^t \times M_2^t} d_t\right|_{g(t)}^2 &=& g_t\left( \nabla^{M_1^t \times M_2^t} d_t, \nabla^{M_1^t \times M_2^t} d_t\right)\\
%&=& (g_t)_{ij}~g_t^{ai}\del_a d_t \cdot g_t^{bj}\del_b d_t \\
%&=& g_t^{ij} \del_i d_t \cdot \del_j d_t.
%\end{eqnarray}
Therefore, 
\be \Delta_{M_1^t \times M_2^t} D^t = \frac{d_t}{D^t} \Delta_{M_1^t \times M_2^t} d_t + \frac{\left | \nabla^{M_1^t \times M_2^t} d_t\right|_{g(t)}^2}{D^t} - \frac{1}{D^t} g^{ij}_t \del_i D^t \del_j D^t.\ee
Noting that $\del_i D^t = \frac{d_t}{D^t} \del_i d_t$ thus, we can further simplify the last term in the expression above  to give
\be \label{Del eqn}\Delta_{M_1^t \times M_2^t} D^t = \frac{d_t}{D^t} \Delta_{M_1^t \times M_2^t} d_t + \frac{\left | \nabla^{M_1^t \times M_2^t} d_t\right|_{g(t)}^2}{D^t} - \frac{(d_t)^2}{(D^t)^3} g^{ij}_t \del_i d^t~ \del_j d^t.\ee
Furthermore, since $\frac{\del}{\del t}D^t= \frac{1}{D^t} \left(L\frac{\del L}{\del t} + d_t \frac{\del d_t}{\del t} \right)$ and using (\ref{Del eqn}), the inequality (\ref{d evolution}) can be written  
\be  L(t) \frac{\del L}{\del t} + d_t \frac{\del d_t}{\del t} \geq d_t \Delta_{M_1^t \times M^t_2} d_t + \left | \nabla^{M_1^t \times M_2^t} d_t\right|_{g(t)}^2 -  \left(\frac{d_t}{D^t} \right)^2 g^{ij}_t \del_i d_t  \del_j d_t,\ee
which simplifies as
\be \label{new d evolution} L(t) \frac{\del L}{\del t} + d_t \frac{\del d_t}{\del t} \geq d_t \Delta_{M_1^t \times M^t_2} d_t + \left(1 -  \left(\frac{d_t}{D^t} \right)^2\right) \left| \nabla^{M_1^t \times M_2^t} d_t\right|_{g(t)}^2;\ee
where we used the fact that, for a local basis of tangent vectors $\del_i$ on $M_1 \times M_2$,
\begin{eqnarray}
\left| \nabla^{M_1^t \times M_2^t} d_t\right|_{g(t)}^2 &=& g_t\left( \nabla^{M_1^t \times M_2^t} d_t, \nabla^{M_1^t \times M_2^t} d_t\right)\\
&=& (g_t)_{ij}~g_t^{ai}\del_a d_t \cdot g_t^{bj}\del_b d_t \\
&=& g_t^{ij} \del_i d_t \cdot \del_j d_t.
\end{eqnarray}

To make this construction more explicit, consider
\begin{ex} \label{Ex-torus}
$M_1 \cong M_2 \cong $ the flat torus $\mathbb{T}^2$.
\end{ex}
Let $(M_1, g_1(t)) \cong (M_2, g_2(t)) \cong (\mathbb{R}^2 / \mathbb{Z}^2, g_{\torus}) \cong \left(\sphere^1 \times \sphere^1, \left(\frac{1}{2\pi}\right)^2 dx^2 + \left(\frac{1}{2\pi}\right)^2 dy^2\right)$. Since $\Ric_{\torus} \equiv 0$, the flat torus is a stationary point for the Ricci flow and thus, $g_i(t) \equiv g_{\torus}$, for all $t $, and $i=1,2$. 
Define a family of metrics $D^t : \MM \times \MM \to \R^{\geq 0}$ by setting
\be D^t(\veca, \vecb) = \begin{cases} 
d_{\torus}(\veca, \vecb), \hspace{.95in} \text{ if } \veca, \vecb \in M_i\\
\\
\sqrt{L^2(t) + d_{\torus}^2(\phi(\veca), \vecb)}, \quad \text{ if } \veca \in M_1, \vecb \in M_2 ~\text{ or } ~\veca \in M_2, \vecb \in M_1, 
\end{cases}
\ee
where $\phi:\torus \to \torus$ is the identity map. To interpret (\ref{d evolution}), we consider $D^t$ and $d_t$ as functions on $\torus \times \torus$ with canonical metric  
\be g_{\torus \times \torus} = g_{\torus} \times g_{\torus} = \left(\frac{1}{2\pi}\right)^2 \left[ \left(dx^1\right)^2 + \left(dx^2\right)^2+ \left(dx^3\right)^2 + \left(dx^4\right)^2 \right].\ee
Furthermore, since the metics are stationary, the Laplacian $\Delta_{M^t_1 \times M^t_2}$ is independent of $t$ and $d_t = d$ so that, inside the cut locus,
\be \Delta_{\torus \times \torus} d(\veca, \vecb)= \left. \Delta^1_{\torus} d(\cdot, \vecb)\right|_{\veca} + \left. \Delta^2_{\torus} d(\veca, \cdot)\right|_{\vecb} = \frac{(2\pi)^2}{d_{\vecb}(\veca)} + \frac{(2\pi)^2}{d_{\veca}(\vecb)} = \frac{2(2\pi)^2}{d(\veca, \vecb)};\ee
and thus, 
\be
d_t \Delta_{M_1^t \times M_2^t} d_t =  d \cdot \Delta_{\torus \times \torus} d =  d\cdot \frac{2(2\pi)^2}{d} = 2(2\pi)^2.
\ee
Also, since $g^{ij}_t = g^{ij}$, we have 
\be g_t^{ij} \del_i d_t  \del_j d_t = g^{ij} \del_i d  \del_j d = (2 \pi)^2\sum_{i=1}^4 (\del_i d)^2 = 2(2\pi)^2.\ee
Lastly, since $\nabla^{\torus \times \torus}d = g^{ij} \del_i d \del_j = (2\pi)^2 \delta^{ij} \del_i d~ \del_j=\sum_{i=1}^4(2\pi)^2 \del_i d ~\del_i$, we get
\begin{eqnarray} \left | \nabla^{M_1^t \times M_2^t} d_t\right|_{g(t)}^2 = \left | \nabla^{\torus \times \torus} d\right|^2 &=& g_{ij} \left(\nabla^{\torus \times \torus}d \right)^i \left(\nabla^{\torus \times \torus}d \right)^j \\
&=& (1/2\pi)^2\delta_{ij} \left(\nabla^{\torus \times \torus}d \right)^i \left(\nabla^{\torus \times \torus}d \right)^j\\
&=&\sum_{i=1}^4 (1/2\pi)^2 \left[\left(\nabla^{\torus \times \torus}d \right)^i\right]^2\\
&=& \sum_{i=1}^4(1/2\pi)^2 \left[(2\pi)^2 \del_i d\right]^2 \\
&=&\!\! \sum_{i=1}^4 (2\pi)^2 (\del_i d)^2 \!= \!\left| \nabla^1 d\right|^2\!\! +\! \left| \nabla^2 d\right|^2= 2(2\pi)^2.
\end{eqnarray}
Therefore, in this setting (\ref{new d evolution}) becomes 
\be \label{d evolution for torus} L(t) \frac{\del L}{\del t} \geq 4(2\pi)^2 -2(2\pi)^2\left(\frac{d}{D^t}\right)^2.\ee
Roughly estimating $0 \leq  d/D^t \leq 1$, we can take $L(t) = 2\pi \sqrt{8t + L^2(0)}$ so that $L \frac{\del L}{\del t} = 4(2\pi)^2$ which satisfies (\ref{d evolution for torus}). Naturally, any $L(t)$ with growth larger than $t^{1/2}$ would also satisfy condition (\ref{d evolution}) as well and give another family of distance metrics satisfying the super Ricci flow on the disjoint union. 
\begin{rmk}In general, for $M_1 \cong M_2 \cong (\mathbb{R}^n / \mathbb{Z}^n, g_{{\mathbb{T}}^n})$ we would have\\ $L(t) = 2\pi \sqrt{4 n t + L^2(0)}$.
\end{rmk}
%This can be easily generalized to general Einstein metrics; i.e. metrics of the form $\Ric(g(0)) = \lambda g(0)$. In the example above, $\lambda = 0$ when $\lambda >0$ we have, for example,  
\begin{ex}\label{Ex-sphere}
$M_1 \cong M_2\cong$ the round sphere $\sphere^2$.
\end{ex}
Let $(M_1, g_1(0)) \cong (M_2, g_2(0)) \cong (\sphere^2, g_{can})$, the unit 2-sphere with its canonical round metric. For $i =1,2$ we have $\Ric(g_i(0)) = g_{can}$ so the metrics on $M_i$ evolving by (\ref{RF}) satisfy $g_i(t) = \left(1-2t \right)g_{can}$, for $t \in \left[\left.0, \frac{1}{2}\right.\right)$. We will often write $\sphere^2_t$ to denote $\left(\sphere^2,(1-2t) g_{can}\right)$ and $d$ for the distance metric induced by $g_{can}$.

As before, any family of distance metrics $D^t$ on $\MM$ must satisfy $\left.D^t\right|_{M_i} = d_{g_i(t)}$ and for $\veca \in M_1$, \mbox{$\vecb \in M_2$}, we take
\be D^t(\veca, \vecb) = \sqrt{L^2(t) + d^2_t(\phi(\veca), \vecb)} , \ee
where, as before, $\phi: \sphere^2 \to \sphere^2$ is the identity map and $d_t$ denotes the distance metric on $\sphere^2_t$. That is to say, $d^2_t =d^2_{g_i(t)} = \left(1-2t\right) d^2_{g_i(0)} = \left(1-2t\right) d^2$.

Furthermore, in this setting we have
\be d_t \frac{\del d_t}{\del t} = \sqrt{1-2t}~d \frac{\del}{\del t}\sqrt{1-2t} d= -d^2\ee
and, since $\Delta_{\sphere^2_t \times \sphere^2_t} = \frac{1}{1-2t}\Delta_{\sphere^2 \times \sphere^2}= \frac{1}{1-2t} \left[ \Delta^1_{\sphere^2}+ \Delta^2_{\sphere^2}\right]$, we have
\begin{eqnarray}
d_t \Delta_{\sphere^2_t \times \sphere^2_t} d_t = \sqrt{1-2t} d \Delta_{\sphere^2_t \times \sphere^2_t} \sqrt{1-2t} d
 &=& (1-2t)d~ \Delta_{\sphere^2_t \times \sphere^2_t} d \\
&=&\!\! d~\Delta_{\sphere^2 \times \sphere^2} d\\
&=&\!\! d\Delta^1_{\sphere^2}d + d \Delta^2_{\sphere^2} d\\
%&=&\!\! d \cot d + d \cot d\\
&=&\!\! 2d\cot d,
\end{eqnarray}
where in the third line 
\be \Delta^1_{\sphere^2}d(\veca, \vecb) = \left. \Delta^1d(x,y)\right|_{x = \veca, y = \vecb}= \left.\Delta^1_{\sphere^2}d(x, \vecb)\right|_{x= \veca}=\left.\Delta^1d_{\vecb}(x)\right|_{x = \veca}= \cot d_ {\vecb}(\veca) = \cot d(\veca, \vecb). \ee
Also, we have
\begin{eqnarray} \left| \nabla^{\sphere^2_t \times \sphere^2_t} d_t \right|_{g(t)}^2 = (1-2t)\left| \nabla^{\sphere^2_t \times \sphere^2_t} d_t \right|_{can}^2 &=& \left|\nabla^{\sphere^2 \times \sphere^2} d \right|_{can}^2\\
&=& \left|\nabla^1 d \right|_{can}^2 + \left|\nabla^2 d \right|_{can}^2;
\end{eqnarray}
%which follows since
%\be g^{ij}_t \del_i d_t \del_j d_t = \frac{1}{1-2t} g^{ij} \del_i \sqrt{1-2t} d~ \del_j \sqrt{1-2t} d=g^{ij} \del_i d ~\del_j d\ee
thus, the expression for (\ref{new d evolution}) in this setting can be written as 
\be \label{sphere}L(t) \frac{\del L}{\del t} -d^2 \geq 2d \cot d+ \left(1- \left(\frac{d_t}{D^t} \right)^2\right) \left(\left|\nabla^1 d \right|_{can}^2 + \left|\nabla^2 d \right|_{can}^2 \right).\ee
Keeping in mind $0 \leq d \leq \pi$, any $L(t)$ which satisfies (\ref{sphere}) for all $t \in [0, 1/2)$ gives a suitable distance metric on $\sphere^2 \times \sphere^2$. This can also be extended to higher dimensional spheres in the obvious way.
%\begin{Example}
%Solitons
%\end{Example}

%\begin{Example}
%Hyperbolic space $\mathbb{H}^n$
%\end{Example}
%\red{Here we can address the problem of when the distance fails to remains a distance metric. Mention that we take $\inf_{x\in M_1, y \in M_2} D^t(x,y) \geq \max \{\diam(M_1), \diam(M_2)\}$}

\begin{rmk} A variation of this construction can be used for $M_1$ and $M_2$ which are only  assumed to be homeomorphic. In the definition of $D^t$ given in (\ref{Defn-D^t}), take, for  $ \veca \in M_1, \vecb \in M_2$
$$\min \left(\inf_{\phi} \sqrt{L^2(t) + d_{g_2(t)}^2(\phi(\veca), \vecb))}, ~~ \inf_{\phi}\sqrt{L^2(t) + d_{g_1(t)}^2(\veca, \phi^{-1}(\vecb))}\right), $$
where the infimum is taken over all homeomorphisms $\phi: M_1 \to M_2$ and, as before, $L(t)$ depends only on $t$.
\end{rmk}

\section{Background}\label{Section-Background}
As we hope to make clear, our current results tie together a progression of ideas which originated with a 2005 paper by M. von Renesse and K.T. Sturm \cite{vRS2005}, although its true origins can be recognized in earlier work of Bakry-Emery \cite{BakryEmery1985}, Cordero-Erausquin, McCann, Scmukenschlager \cite{Cordero-ErausquinMcCannSchmuckenschlager2001,Cordero-ErausquinMcCannSchmuckenschlager2006} and others.

\subsection{Metric characterizations of Ricci curvature lower bounds and the Ricci flow}
In \cite{vRS2005}, von Renesse-Sturm characterize uniform lower Ricci curvature bounds of smooth Riemannian manifolds $(M^n,g)$ using various convexity properties of the entropy as well as transportation inequalities of volume measures, heat kernels, and gradient estimates of the heat semigroup on $M^n$. In fact, the metric nature of the ideas presented in that paper introduced into the literature a discussion of so called ``synthetic'' definitions of Ricci curvature lower bounds which do not rely on the underlying smooth structure of the manifold and thus lend themselves to spaces lacking that smooth structure, such as metric measure spaces, Alexandrov spaces, or general metric spaces.

We state here only a small part the results in \cite{vRS2005} which are relevant to our later discussion. First a bit of notation: Let $(M^n, g)$ be a smooth, connected, complete Riemannian manifold of dimension $n$. Denoting the heat kernel on $M^n$ by $p_t(x,y)$ one can define the operators $p_t: C^{\infty}_c (M) \to C^{\infty}(M)$ and $p_t: L^2 \to L^2 (M)$ by $f \mapsto p_tf(x) : = \int_Mp_t(x,y) f(y) ~d\vol(y)$.
They prove
\begin{thm}\label{thm-vonRenesseSturm}{\em (von Renesse-Sturm, \cite{vRS2005})} For any smooth, complete Riemannian manifold $(M^n, g)$ %and any $K \in \R$, 
$\Ric_M \geq 0$ if and only if for all bounded $f \in \mathcal{C}^{Lip}(M)$ and all $t >0$, \[\Lip(p_tf) \leq \Lip(f)\]
\end{thm}
%Indeed, both K.T. Sturm in \cite{Sturm2006I, Sturm2006II} and J. Lott and C. Villani in \cite{LottVillani2009} published papers proposing a definition of Ricci curvature bounds for arbitrary metric measures spaces (see also work of Ohta, Kuwae-Shioya, et al. for closely related definitions) and proving a number of preliminary results well-known for smooth Riemannian manifolds with lower Ricci curvature bounds including, among other things, a Bishop-Gromov volume comparison theorem and a  Myer's diameter bound theorem. 

%Although introduced independently, their definitions each involved notions of optimal transport particularly by requiring the convexity of a certain functional as a measure is transported along a path in the Wasserstein space $(P^2(M), d^W_2)$.
Later McCann-Topping \cite{McCT2010} took a dynamic approach and reinterpreted the work of von Renesse-Sturm in relation to a metric evolving by the Ricci flow. Specifically, they characterize super solutions of the Ricci flow on $M^n$ by the contractivity of mass diffusions backwards in time $t$. We refer to a super solutions of the Ricci flow as a super Ricci flow. That is 
\begin{defn}\label{Def-superRF}\textbf{(McCann-Topping; c.f. \cite{McCT2010}, Definition 1).}
For a compact, oriented $n$-dimensional manifold, a super Ricci flow is a smooth family $g(t)$ of metrics on $M$, $t \in [0,T]$, such that at each $t \in (0,T)$ and each point on $M$, one has
\be \label{superRF}  \frac{\partial g}{\partial t} + 2 \Ric(g(t)) \geq 0.\ee
\end{defn}
In addition, and more closely related to our purposes, they prove the following
\begin{thm}\label{Thm-McCT}\textbf{(McCann-Topping; c.f. \cite{McCT2010}, Theorem 2).}
Let $M^n$ be a compact, Riemannian manifold of dimension $n$. A smooth one-parameter family of metrics for $t \in [0,T)$ is a super Ricci flow if and only if whenever $0<a<b<T$ and $f: M \times (a,b) \to \R$ is a solution to $\frac{\del f}{\del t} = \Delta_{g(t)}f$, then 
\[\Lip(f, t):= \sup_{x\neq y} \frac{|f(x, t) - f(y,t)|}{d(x,y,t)} \text{ is non-increasing in } t.\] 
\end{thm}
The quantity $\Lip(f,t)$ is the Lipschitz constant of $f(\cdot, t)$ evaluated using the metric $g(t)$. It is precisely this characterization which we use to define the notion of a super Ricci flow for the disjoint union of two evolving Riemannian manifolds. However, we must first make sense of the local representation for the heat kernel on $\MM$ in order to describe what it means for a function $u(x,t)$ on $\MM$ to solve the heat equation.

\subsection{Heat kernel operators from the metric and measure}

In \cite{vRS2005}, von Renesse-Sturm focus on smooth, connected complete $n$-dimensional Riemannian manifolds $M^n$ and characterize a uniform lower Ricci curvature bound of $M^n$ using, among other things, heat kernels and transportation inequalities for uniform distribution measures on distance spheres in $M^n$. One striking advantage of these characterizations is that they depend only on the metric and measure of the underlying smooth Riemannian manifold and thus allow for a notion of a Ricci curvature lower bound depending solely this basic, non-smooth data. In fact, these characterizations ultimately led to the current definitions of Ricci curvature for arbitrary metric measure spaces introduced independently by Lott-Villanni and Sturm \cite{LV2009, St2006I, St2006II}. We recall now the original discussion of von Renesse-Sturm.

Following the comment at the end of Section 1 of \cite{vRS2005}, one can view a smooth, connected Riemannian manifold $(M,g)$ as a separable metric measure space $(M, d_g, \vol_g)$ and define a family of Markov operators $\sigma_r$ acting on the set of bounded Borel measurable functions  by $\sigma_rf(x) = \int_M f(y)~ d\sigma_{r,x}(y)$, where the measure $\sigma_{r,x}$ is defined as 
\be \label{sigma-rx} \sigma_{r,x} (A) := \frac{\vol_g(A \cap \partial B(x,r))}{\vol_g(\partial B(x,r))}, \quad A \in \mathcal{B}(M).
\ee
Here $B(x,r)$ denotes the ball of radius $r$ centered at $x$. By the Arzela-Ascoli theorem and applying the Trotter-Chernov product formula \cite{EK1986}, there exists a subsequence such that for all bounded $f \in C^{\Lip}(M)$ the limit 
\be \label{heatsoln} p_tf(x) := \lim_{j \to \infty} \left(\sigma_{\sqrt{2nt/j}} \right)^j f(x)\ee
exists and converges uniformly in $x \in M$ and locally uniformly in $t \geq 0$. In fact, if we let $\texttt{p}_t(x,y)$ denote the minimal smooth heat kernel on $M^n$ (i.e. the positive fundamental solution to $(\Delta - \frac{\del}{\del t})\texttt{p}_t(x,y) = 0$) then it follows that $\texttt{p}_tf(x) = p_tf(x)$. Thus, by  (\ref{heatsoln}) we describe solutions to the heat equation for an arbitrary metric measure space $(M,d,m)$ without relying on a smooth structure.

\subsection{Constructing a heat kernel on $M_1 \sqcup M_2$} 
Now we return the dynamic situation and consider a single smooth manifold evolving by the Ricci flow. Take $g(t)$ a family of metrics on $M$ satisfying (\ref{RF}) for $t \in [0, T)$, $T>0$. At each time $t$, just as in (\ref{sigma-rx}), define the normalized Riemannian uniform distribution on spheres centered at $x \in (M, g(t))$ of radius $r>0$ by \be \label{sigmatrx} \sigma^{t}_{r,x} (A) := \frac{\mathcal{H}^{n-1}(A \cap \partial B^{t}(x,r))}{\mathcal{H}^{n-1}(\partial B^{t}(x,r))}, \quad A \in \mathcal{B}(M),
\ee
where $B^{t}(x,r)$ denotes the ball of radius $r$ centered at $x$ with respect to the fixed  metric $g(t)$. As before, we have a family of Markov operators $\sigma^{t}_r$ on the set of bounded Borel-measurable functions $(M, g(t))$ defined above replacing $\sigma_r$ by $\sigma^t_r$ and integrating over $(M,g(t))$. Just as before, we have (for a subsequence)
\be \label{sigtconv} \left(\sigma^{t}_{\sqrt{2n\tee/j}}\right)^j f(x) \to p^{t}_\tee f(x) = e^{\tee \Delta_{g(t)}} f(x) \ee  
uniformly in $x \in (M,g(t))$ and locally uniformly in $\tee\geq 0$ for all bounded $f \in \mathcal{C}^{\text{Lip}}(M^n, g(t))$. 

Consider now the entire space time where the Ricci flow is defined for $M$; i.e. $M \times [0, T)$. Let $B$ denote the Banach space $\mathcal{C}^{\text{Lip}}(M^n, g(t))$ with the sup-norm and $\mathcal{L}(B)$ the space of bounded linear operators on $B$. For each  $t$, consider functions $F_{t}: [0, \infty) \to \mathcal{L}(\mathcal{B})$ where  
\be \label{t operators} F_{t}(\tee) = e^{\tee \Delta_{g(t)}}.\ee
Note that $F_t(0) = Id$ for every $t \in [0,T)$ and %furthermore there exists a constant $c\geq0$
%\be \sup_{t \in [0, T)} \left|\left| F_t(\texttt{t})\right|\right|_{\infty} \leq e^{c\texttt{t}}\ee
%for the supremum-norm of $\mathcal{L}(B)$.
%Also, noting that 
for any $f \in B$
\begin{eqnarray*}
F'_{t}(0)f &=& \lim_{\tee \downarrow 0} \frac{F_{t}(\tee)f - f}{\tee} =  \lim_{t \downarrow 0} \frac{e^{\tee \Delta_{g(t)}}f - f}{\tee} \\
&=& \Delta_{g(t)}f.
\end{eqnarray*}
Thus, by applying a generalization of the Trotter-Chernov product formula (\cite{V2010}, Main Theorem) to the time-dependent  operators of (\ref{t operators}), for any function $u : M \times (0,T) \to \mathbb{R}$ solving the initial value problem 
\be \label{heat IVP}
\begin{cases}
\dfrac{d}{dt}u(\x,t) = \Delta_{g(t)}u(\x,t) \\
u(\x,0) = f(\x),
\end{cases}
\ee
for which there exists a corresponding one-parameter family of bounded linear operators $U(t, 0)_{0 \leq t \leq T}$ in $B$ such that  $u(\x,t) = U(t,0)f(\x)$, it follows that for all $0 \leq t \leq T$ we have
\be \label{general Trotter}
U(t,0) = \lim_{m\to \infty} \prod^0_{i = m-1}F_{\frac{i}{m}t}\left(\frac{t}{m}\right) = \lim_{m\to \infty} \prod^0_{i = m-1} e^{\frac{t}{m}\Delta_{g\left(\frac{i}{m}t\right)}}
\ee
with convergence of the limit in the strong operator topology of $\mathcal{L}(B)$. Combining this (\ref{sigtconv}) we can further write, for any $f \in B$,
\be \label{RF IVP soln} u(\x, t) = U(t,0) f(\x) = \lim_{m\to \infty} \prod^0_{i = m-1} \lim_{j \to \infty} \left(\sigma^{\frac{i}{m}t}_{\sqrt{\frac{2nt}{jm}}}\right)^j f(\x).\ee 
Naturally, as we saw earlier, this description gives a metric measure characterization of solutions to the  heat equation on the evolving manifold $(M,g(t))$. 

Finally, we turn our attention to the situation of the current paper and use the characterization above to describe solutions for the heat equation on $\MM$. Note that the description in (\ref{RF IVP soln}) is locally defined and thus allows for generalization to the disjoint union $\MM$. Indeed, as $j \to \infty$ the operators  $\sigma^{\frac{i}{m}t}_{\sqrt{\frac{2nt}{jm}}}$ are ultimately restricted to individual components $M_1$ or $M_2$ of $\MM$ depending on whether $x \in M_1$ or $x \in M_2$ (resp.). Motivated by the discussion above we define

\begin{defn} \label{disjoint heat soln}
Let $(M_i, g_i(t))$, for $i=1,2$, be compact Riemannian manifolds supporting smooth families of metrics satisfying the Ricci flow equation given by (\ref{RF}) for $t \in [0, T_i)$. Also, let $D^t$ be a family of distance functions on $\MM$ so that each $t \in [0, \min(T_1,T_2))$ we have $(\MM, D^t)$ is a complete, compact metric space compatible with the family of metrics $g_i(t)$ on $M_i$ resp.; i.e. for $i=1,2$,
\be \label{restricted d} \left. D^t \right|_{M_i} = d_{g_i(t)},\ee
and such that
\be\frac{\partial}{\partial t}D^t(x,y) \geq \Delta_{M_1^t \times M_2^t}D^t(x,y), \quad  \text{for }~x \in M_1, y \in M_2. \ee
A function $u: M_1 \sqcup M_2 \times (0,T) \to \mathbb{R}$ is said to solve the initial value problem (\ref{heat IVP}) on $\MM$ for $f \in \mathcal{C}^{\text{Lip}}\left(M_1 \sqcup M_2, D^t\right)$, provided
\be u(\x, t) = \lim_{m\to \infty} \prod^0_{i = m-1} \lim_{j \to \infty} \left(\sigma^{\frac{i}{m}t}_{\sqrt{\frac{2nt}{jm}}}\right)^j f(\x).\ee 
\end{defn}
Note that 
\begin{lemma}\label{u is heat}
Let $(M_i, g_i(t))$, for $i=1,2$, and $(\MM, D^t)$ be as above and suppose $D^0(x,y)  >0$ for all $x\in M_1, y\in M_2$. A function $u: \MM \times (0,T) \to \R$ solves the initial value problem (\ref{heat IVP}) on $M_1 \sqcup M_2$ if and only if $\left.u\right|_{M_i}$ and satisfies smooth heat equation on $M_i$, for $i = 1,2$.
 \end{lemma}
 
 \begin{proof} First, note that if $D^0(x,y) >0$ for $x \in M_1, y\in M_2$ at the initial time $t=0$, then  by the maximum principle (see, for example, Theorem 3.1.1 of \cite{T2006}) we have $$D^t(x,y) >0, \text{ for all } t >0 \text{ and } x \in M_1, y \in M_2.$$ For a fixed $t$, it follows that the measures $\sigma^t_{r,x}$ when defined on $M_1 \sqcup M_2$ agree with $\left. \sigma^t_r \right|_{M_i}$ for $x \in M_i$ provided $r$ is taken small enough; namely $r < \inf_{x\in M_1, y\in M_2} D^t(x,y)$.  Thus, for $j$ large enough it follows that 
 \be \label{sigma restriction} \sigma ^{\frac{i}{m}t}_{\sqrt{\frac{2nt}{jm}}} =  \left.\sigma ^{\frac{i}{m}t}_{\sqrt{\frac{2nt}{jm}}}\right|_{M_i}.\ee
Now for $u: \MM \times (0,T) \to \R$ which satisfies the IVP given in (\ref{heat IVP}) we have that
\be \left.u(\x, t)\right|_{M_1} = \lim_{m\to \infty} \prod^0_{i = m-1} \lim_{j \to \infty} \left(\left.\sigma^{\frac{i}{m}t}_{\sqrt{\frac{2nt}{jm}}}\right|_{M_1}\right)^j f(\x).\ee
As pointed out in the discussion above, for a smooth Riemannian manifold $M_1$ whose heat kernel is  denoted by $\texttt{p}_{t}(x,y)$, since $\texttt{p}_{t}f(x) = p_tf(x)$, we have $$\texttt{p}_tf(x) = \lim_{j \to \infty} \left(\sigma_{\sqrt{2nt/j}} \right)^j f(x).$$ Thus, we can write using the notation as before where $F_t(\texttt{t})=e^{\texttt{t} \Delta_{g(t)}} = p_{\texttt{t}}^t$,
\be \left.u(\x, t)\right|_{M_i} = \lim_{m\to \infty} \prod^0_{i = m-1} \texttt{p}^{\frac{i}{m}t}_{\frac{t}{m}}f(\x)=\lim_{m\to \infty} \prod^0_{i = m-1} F_{\frac{i}{m}t}\left(\frac{t}{m}\right) f(\x) = U(t,0)f(\x).\ee
Thus, by the generalized Trotter product formula and (\ref{general Trotter}), it follows that $u(\x,t)|_{M_1}$ solves the heat equation on $M_1$. In precisely the same way, we verify that $\left.u\right|_{M_2}$ also satisfies the heat equation on $M_2$.

Furthermore, suppose some function $u(\x, t)$ defined on $\MM$ when restricted to either $M_i$ satisfies the heat equation on that component. Again by (\ref{sigma restriction}) it follows that $u(\x,t)$ satisfies the IVP on the disjoint union $\MM$.
\end{proof}
 
\section{Proof of Theorem \ref{MainThm} and consequences}\label{Section-Proof/Consequences}
\begin{proof} (Theorem \ref{MainThm}).
With $M_i$ as above, let $u_i : M_i \times (0,T)$ be solutions to $\dfrac{\partial u_i}{\partial t} = \Delta_{g_i(t)}u_i$, $i = 1,2$. Consider the disjoint union $M_1 \sqcup M_2$ and define a function $u: M_1 \sqcup M_2 \times (0,T) \to \R$ by
\be
u(x, t) = \begin{cases}
u_1(x,t), ~\text{ when } x \in M_1  \\
u_2(x,t), ~\text{ when } x \in M_2.
\end{cases}
\ee
Recall, by assumption \be \label{posdt}D^t(m_1,m_2) > 0, ~\text { for all }~ m_1 \in M_1, m_2 \in M_2, t >0,\ee so by Lemma \ref{u is heat}, the function $u(x,t)$ satisfies the heat equation on $M_1 \sqcup M_2$. Note that for any $t\in[0,T)$, there exists $p, q \in (M_1 \sqcup M_2, D^t)$ such that 
\be \Lip(u, t) = \frac{|u(p,t)-u(q,t)|}{D^t(p, q)}. \ee
Clearly, if $p, q \in M_i$, fixed, then by Theorem \ref{Thm-McCT} of Topping-McCann, the property that $\Lip(u, t)$ is non-increasing as a function of $t$ is equivalent to $g_i(t)$ being a solution to the super Ricci flow. Thus, we are done since each $(M_i, g_i(t))$ in fact solves (\ref{RF}) by assumption and so obviously (\ref{superRF}). Therefore, we focus on the case when the Lipschitz constant of $u$ is achieved by a point in $M_1$ and a point in $M_2$. 

Fix $t\in (0,T)$. Without loss of generality, assume the value of $\Lip(u,t)$ is attained by the points $p\in M_1, q\in M_2$. In a neighborhood sufficiently near $(p,q) \in M_1 \times M_2$, we may also assume (without loss of generality) that $u_1(x,t) - u_2(y,t) \geq 0$ so that the function on $M_1\times M_2$ given by \[(x,y) \mapsto \frac{u_1(x,t) - u_2(y,t)}{D^t(x,y)} \] is nonnegative and has an absolute maximum at the point $(p,q)$. Therefore,

\be
\label{gradzero} \left. \nabla\left(\frac{u_1(x,t) - u_2(y,t)}{D^t(x,y)}\right)\right|_{(p,q)}=0,
\ee
and
\be
\label{Delnonpos} \left. \Delta \left(\frac{u_1(x,t) - u_2(y,t)}{D^t(x,y)}\right)\right|_{(p,q)} \leq 0.
\ee
Furthermore, for points $x,y \in M_1 \sqcup M_2$ sufficiently close to $p \in M_1$ and $q \in M_2$ (resp.) it follows from (\ref{posdt}) that $u_1(x,t) - u_2(y,t) = \left. u\right|_{M_1}(x,t) - \left. u\right|_{M_2}(y,t) = u(x,t) - u(y,t)$. 

To simplify notation, set $\ubar(x,y,t) =u_1(x,t) - u_2(y,t)$. From (\ref{gradzero}) we have
\be \nabla \left(\frac{\ubar}{D^t}\right) = \frac{D^t\nabla \ubar - \ubar \nabla D^t}{\left(D^t\right)^2} = 0,\ee
and thus
\be \label{u nabla D^t} \ubar \nabla D^t = D^t \nabla \ubar.\ee
To evaluate (\ref{Delnonpos}), note that
\begin{eqnarray}
\nabla^2\left(\frac{\ubar}{D^t}\right) &=& \frac{\left(D^t\right)^2\left(\nabla D^t \nabla \ubar + D^t \nabla^2 \ubar - \nabla \ubar \nabla D^t - \ubar \nabla ^2 D^t\right) - 2 D^t\nabla D^t(D^t\nabla \ubar - \ubar \nabla D^t)}{\left(D^t\right)^4} \hspace{.4in} \\
&=&\frac{\nabla^2\ubar}{D^t} - \ubar \frac{\nabla ^2 D^t}{\left(D^t\right)^2} - \frac{\nabla D^t \otimes \nabla \ubar}{\left(D^t\right)^2} - \frac{\nabla \ubar \otimes \nabla D^t}{\left(D^t\right)^2} + 2 \frac{\ubar \nabla D^t \otimes \nabla D^t}{\left(D^t\right)^3};
\end{eqnarray}and, therefore
\begin{eqnarray}
\Delta\left( \frac{\ubar}{D^t}\right) &=&\tr \nabla^2\left(\frac{\ubar}{D^t} \right)\\
 &=& \tr \frac{\left(D^t\right)^2\left( \nabla D^t \nabla \ubar + D^t \nabla^2 \ubar - \nabla \ubar \nabla D^t - \ubar \nabla^2 D^t\right) - 2D^t \nabla D^t (D^t \nabla \ubar - \ubar \nabla D^t)}{\left(D^t\right)^4} \hspace{.4in}\\
&=& \frac{\Delta \ubar}{D^t} - \frac{\ubar \Delta D^t}{\left(D^t\right)^2} - 2 \frac{\langle\nabla D^t, \nabla \ubar \rangle}{\left(D^t\right)^2} + 2 \frac{\ubar\left| \nabla D^t \right|^2}{\left(D^t\right)^3},
\end{eqnarray}
where we used (\ref{u nabla D^t}) to evaluate in the last term. Furthermore, using (\ref{u nabla D^t}) to write $\nabla \ubar = \frac{\ubar \nabla D^t}{D^t}$, we have
\be 2\frac{\langle \nabla D^t, \nabla \ubar,\rangle }{\left(D^t\right)^2} = \frac{\langle \nabla D^t, \frac{\ubar \nabla D^t}{D^t}\rangle }{\left(D^t\right)^2} = 2 \frac{\ubar \left|\nabla D^t \right|^2}{\left(D^t\right)^3}
\ee
which implies 

\be \left. \Delta\left( \frac{\ubar}{D^t}\right)\right|_{(p,q)} = \frac{\Delta \ubar}{D^t} - \frac{\ubar \Delta D^t}{\left(D^t\right)^2}.\ee
So, by (\ref{Delnonpos}), it follows that at $(p,q)$
\be \frac{\Delta \ubar}{D^t} \leq \frac{\ubar \Delta D^t}{\left(D^t\right)^2};\ee
or, equivalently,
\be \Delta \ubar (p,q) \leq \frac{\ubar(p,q)}{D^t(p,q)} \Delta D^t(p,q). \ee

By assumption, $\frac{\partial}{\partial t} D^t \geq \Delta D^t$, and since $\frac{\ubar}{D^t} \geq 0$ we get
\be \Delta \ubar \leq \frac{\ubar}{D^t} \frac{\partial D^t}{\partial t},\ee
and, thus, since $\ubar: \left(M_1 \times M_2\right) \times (0,T)\to \mathbb{R}$ solves the heat equation by Lemma \ref{u is heat},  
\be 
\label{du/dt} \frac{\partial \ubar}{\partial t}\leq \frac{\ubar}{D^t} \frac{\partial D^t}{\partial t}.
\ee

Finally, note that 
\be
\frac{\partial }{\partial t} \Lip(u,t) = \frac{\partial }{\partial t} \sup_{\substack{x\neq y\\ x \in M_1, y \in M_2}}\frac{|u(x,t) - u(y,t)|}{D^t(x,y)} =  \sup_{\substack{x\neq y\\ x \in M_1, y \in M_2}} \frac{D^t \frac{\partial \ubar}{\partial t} - \ubar \frac{\partial D^t}{\partial t}} {\left(D^t(x,y) \right)^2}
\ee

Since (\ref{du/dt}) holds for any pair of points which achieve the Lipschitz constant, it follows that $\frac{\partial}{\partial t} \Lip(u,t) \leq 0$ and thus we have $\Lip(u,t)$ is decreasing as a function of $t$ and we are done.
\end{proof}

%\begin{Remark}
%See Lemma 5.3.2 of Topping's notes \cite{T2006}. If the Ricci curvature remains bounded on the individual $M_i$ then the distance function is controlled and the distance function is Lipschitz. This coincides with what we assert here in that we are requiring the distance function, as defined by the distance metric $D^t$ remains Lipschitz \red{check} and ensures that for this Ricci flow of $\MM$ the distance remains controlled as well.
%\end{Remark}

This can be easily generalized to address additional components.
\begin{cor}
For $i = 1,2,\cdots,k$, let $(M_i, g_i(t))$ be compact  $n$-dimensional manifolds whose metrics $g_i(t)$ satisfy (\ref{RF}) for $t \in [0,T_i)$. Consider a family of metric spaces $(M_1 \sqcup M_2 \sqcup \cdots M_k, D^t)$ for $t \in (0,T)$, $T=\min(T_1,T_2,\cdots,T_k)$ and suppose that $D^t$ satisfies (\ref{d evolution}) for all $x \in M_i$, $y\in M_j$ with $i \neq j$, then the family of metrics $D^t$ is a super Ricci flow of $M_1 \sqcup M_2 \sqcup \cdots \sqcup M_k$.
\end{cor}

Furthermore, considering $(\MM, D^t)$ as a family of metric spaces, the evolution inequality given in (\ref{d evolution}) also provides control on how the distance between $M_1$ and $M_2$ changes over time. Namely, if at the initial time $t = 0$ we have $D^0(x,y) \geq c >0$, then  $D^t(x,y) \geq c$ for all $t > 0$. This follows from a direct application of the maximum principle.

S. ~Lakzian\\
{\sc CUNY Graduate Center, 365 Fifth Ave, NY, NY 10016, USA}\\
{\em Email address:} \texttt{slakzian@gc.cuny.edu}

M. ~Munn\\
{\sc University of Missouri, Columbia, MO 65201, USA\\
University of Warwick, Coventry, CV4 7AL, UK\\}
{\em Email address:} \texttt{munnm@missouri.edu}

\end{document}